\documentclass{article}[10pt]
\usepackage{theorem}
\usepackage{eurosym}
\usepackage{amssymb}
\usepackage{float}
\usepackage{amsmath}
\usepackage{amsfonts}
\usepackage[american]{babel}
\usepackage{verbatim}
\usepackage{geometry}

\newtheorem{assumption}{Assumption}
\newtheorem{theorem}{Theorem}

\newtheorem{definition}[theorem]{Definition}
{\theorembodyfont{\rmfamily}
\newtheorem{example}[theorem]{Example}
}

\newtheorem{lemma}[theorem]{Lemma}

\newtheorem{proposition}[theorem]{Proposition}
{\theorembodyfont{\rmfamily}
\newtheorem{remark}[theorem]{Remark}
}

\newenvironment{proof}[1][Proof]{\noindent\textbf{#1} }{\ \rule{0.5em}{0.5em}}

\newcommand*\re{\mathbb{R}}

\newcommand*\Omegabar{\overline{\Omega}}
\newcommand*\delomega{\partial\Omega}
\newcommand*\intdelomega{\int_{\partial\Omega}}
\newcommand*\intomega{\int_{\Omega}}
\newcommand*\diw{\operatorname{div}}
\newcommand*\curl{\operatorname{curl}}

\newcommand*\p{\partial}

\begin{document}
\title{On the boundary conditions in estimating $\nabla \omega$ by $\diw \omega$ and $\curl \omega$}
\maketitle
\date{}

\centerline{\scshape Gyula Csat\'{o}$^1,$  Olivier Kneuss $^2$ and Dhanya Rajendran$^3$}
\medskip
{\footnotesize
\centerline{1 Departamento de Matem\'atica, Universidad de Concepci\'on, Concepcion, Chilel, supported by Fondecyt Grant Nr. 11150017}
\centerline{2 Departamento de Matem\'atica, Universidade Federal do Rio de Janeiro, Brasil.}
\centerline{3 Departamento de Ingener\'ia Matem\'atica, Universidad de Concepci\'on, Concepcion, Chile.}
   \centerline{gy.csato.ch@gmail.com,olivier.kneuss@gmail.com,dhanya.tr@gmail.com}

}

\smallskip

\begin{abstract}
In this paper we study under what boundary conditions the inequality
$$
  \|\nabla\omega\|_{L^2(\Omega)}^2\leq C\left(\|\curl\omega\|_{L^2(\Omega)}^2+ \|\diw\omega\|_{L^2(\Omega)}^2+\|\omega\|_{L^2(\Omega)}^2\right)
$$
holds true. It is known that such an estimate holds if either the tangential or normal component of $\omega$ vanishes on the boundary $\delomega.$ We show that the vanishing tangential component condition is a special case of a more general one. In two dimensions we give an interpolation result between these two classical boundary conditions.
\end{abstract}

\let\thefootnote\relax\footnotetext{\textit{2010 Mathematics Subject Classification.} Primary 26D10, Secondary 35Q30, 35Q61.}
\let\thefootnote\relax\footnotetext{\textit{Gaffney inequality, divergence and curl operators, tangential and normal components, boundary conditions.} }

\section{Introduction}\label{section:introduction}

In this paper we study the estimate
\begin{equation}\label{intro:eq:main}
  \|\nabla\omega\|_{L^2(\Omega)}^2\leq C\left(\|\curl\omega\|_{L^2(\Omega)}^2+ \|\diw\omega\|_{L^2(\Omega)}^2+\|\omega\|_{L^2(\Omega)}^2\right),
\end{equation}
where $\omega\in H^1(\Omega)^n$ is a vector field ($n=2,3$ in most applications) and $C$ is a constant independent of $\omega.$ $H^1(\Omega)^n$ denotes the Sobolev space of vector fields, whose components and all its derivatives are $L^2$ integrable. It is well known that such an estimate holds true if either the tangential or normal component of $\omega$ vanishes on the boundary $\delomega,$ which we shall call the classical boundary conditions. More precisely if $\nu$ is the unit exterior normal vector on $\delomega,$ then \eqref{intro:eq:main} holds true if
\begin{equation}
 \label{intro:eq:classical boundary cond}
  \nu\times\omega=0\quad\text{ on }\delomega\quad\text{ or }\quad\langle\nu;\omega\rangle=0\quad\text{ on }\delomega.
\end{equation}
These boundary conditions have been studied in great detail and the literature on it and its applications to physical systems, mainly Maxwell's equations and Navier-Stokes equations, is very large. Our aim is to show that some of these classical boundary conditions can be extended to much more general ones.
A particular case of our main result  gives in two dimensions an interpolation between the two classical boundary conditions
(cf. Remark \ref{remark:main.theorem.lambda} (ii)).

Let us first mention that inequality \eqref{intro:eq:main} cannot hold true without further restrictions on $\omega.$ To see this take any domain $\Omega\subset\re^2$ and define for  $n\in\mathbb{N}$
$$
  \omega_n(x)=\left(e^{nx_1}\cos(nx_2),-e^{nx_1}\sin(nx_2)\right).
$$
Then one easily verifies that $\diw\omega_n=0,$ $\curl\omega_n=0,$
$$
  |\nabla\omega_n(x)|^2=2n^2e^{2nx_1}\quad\text{ and }\quad|\omega_n(x)|^2=e^{2nx_1}.
$$
Hence there can be no constant $C$ independent of $n$ such that for all $n$
\begin{equation}\label{eq:intro:counterexample}
  2n^2\intomega e^{2nx_1}\leq C\intomega e^{2nx_1}.
\end{equation}
A similar example works also in higher dimensions.

Some standard references on \eqref{intro:eq:main} and its applications are Amrouche-Bernardi-Dauge-Girault \cite{Amrouche-Bernardi...}, Costabel \cite{Costabel 1991}, Dautray-Lions \cite{Dautray-Lions}, and Grisvard \cite{Grisvard Elliptic Pr in nonsmooth}. The inequality \eqref{intro:eq:main} has also been studied in the more general context of differential forms, where $\curl$ is replaced by the $d$ operator, respectively $\diw$ is replaced by $\delta.$ In this setting it is called Gaffney-Friedrichs inequality after Gaffney \cite{Gaffney1}, \cite{Gaffney2}, but for domains with boundary and the classical boundary conditions it is due to Morrey \cite{Morrey 1956}, \cite{MorreyEells}  or Friedrichs \cite{Friedrichs}. Proofs of this general version can also be found in Csat\'o-Dacorogna-Kneuss \cite{Csato-Dacorogna-Kneuss}, Iwaniec-Martin \cite{Iwaniec}, Morrey \cite{Morrey 1966}, Schwarz \cite{Schwarz}, Taylor \cite{Taylor}. Therefore we will call also \eqref{intro:eq:main} Gaffney inequality henceforth.

The first and simplest generalization of the boundary conditions (cf. Theorem \ref{theorem:mixed classical bdry conditions}) is by mixing the classical ones, namely requiring that  on some parts of the boundary the tangential part vanishes and on other parts the normal part vanishes. This result already seems to be known, see for instance Goldsthein-Mitrea \cite{Goldsthein-Mitrea} or Jakab-Mitrea-Mitrea \cite{Mitrea and others1} and the references therein. We state and indicate  a very simple proof of this result for completeness (cf. Theorem \ref{theorem:mixed classical bdry conditions}), since it does not appear explicitly in the references. 

First attempts to give more general boundary conditions have been obtained in  Csat\'o-Dacorogna \cite{Csato-Dacorogna 2010}, see also Csat\'o \cite{Csato Analysis} for a general version on Riemannian manifolds. There the authors have proven in particular that, in three dimensions, if $\lambda$ is a given fixed vector field, then there exsits a constant $C=C(\Omega,\lambda)$ such that \eqref{intro:eq:main} holds true if
$$
  \nu\times\omega=\lambda\langle\nu;\omega\rangle\quad\text{ on }\delomega.
$$
This generalizes the classical condition of vanishing tangential component by setting $\lambda=0.$

Our first main result is an even simpler generalization of this classical boundary condition (vanishing tangential component) which additionally has an obvious geometric interpretation. Namely Theorem \ref{theorem:lambda wedge omega for vector fields} asserts that \eqref{intro:eq:main} holds true if
\begin{equation}\label{eq:intro:lambda times main condition}
  \lambda\times\omega=0\quad\text{ on }\delomega,
\end{equation}
where again $C$ will depend on $\lambda$ and $\Omega.$ Geometrically this means that Gaffney inequality holds true whenever the vector fields $\omega$ are collinear with a given fixed vector field on $\delomega.$ This time, setting $\lambda=\nu$ gives the classical boundary condition. We will prove Gaffney inequality under the condition \eqref{eq:intro:lambda times main condition} for Lipschitz domains as long as $\lambda$ is
$C^1.$ Thus, if $\Omega$ is not $C^{2}$ (and thus $\nu$ is not $C^1$), this result does not include the classical boundary condition $\nu\times\omega=0.$ However, we will give in the case of domains in $\re^2$ a better result which does not even require $\lambda$ to be globally Lipschitz on $\delomega,$
see Theorem \ref{theorem:lambda wedge omega polygonial in re2}. A special case of this theorem is for instance Gaffney inequality on polygonial domains with either of the classical boundary conditions on different parts of the polygon. This is a first step in providing more general Gaffney inequalities, with a simple proof, to be applicable in numerical analyis. We refer to Arnold-Falk-Winther \cite{Arnold Falk Winther} (Section 7.7) and Bonizzoni-Buffa-Nobile \cite{Bonizzoni Buffa Nobile} for a discussion on vector-valued finite element methods and applications of Gaffney inequality in that setting.

We do not require in any of our results that $\Omega$ is convex. This is because we assume that our vector fields $\omega$ are at least in $H^1(\Omega)^n.$ A weaker formulation of the classical Gaffney inequality for Lipschitz domains requires $\Omega$ to be convex. By the weak formulation we mean that we assume
$$
  \omega\in H_T(\diw,\curl;\Omega)=\{\omega\in L^2(\Omega)^n|\, \diw\omega\in L^2(\Omega),\quad\curl\omega\in L^2(\Omega)^n,\quad\nu\times\omega=0\quad\text{ on }\delomega\}.
$$
Under this hypothesis on $\omega,$ Gaffney inequality becomes a regularity result and states that $\omega\in H^1(\Omega)^n$ and satisfies the corresponding estimate \eqref{intro:eq:main}. The same result holds true if we replace $H_T$ by $H_N$, the space with vanishing normal component. The usual approach to prove such regularity results is to use Gaffney inequality for an approximating sequence $\{\omega_k\}$ in $H^1.$ The difficulty consists in establishing $\nu\times\omega_k=0$ on $\delomega,$ using the assumption that $\nu\times\omega=0$ on $\delomega$ in a weak sense. This approximation fails for nonconvex domains, which are only Lipschitz and the regularity statement does not hold true. See for instance the Remark following the proof of Theorem 5.1 in Mitrea \cite{Mitrea Gaffney-Friedrichs ineq}. This is essentially the same example as the one for the Laplace equation: it is well known that the solution $u$ of $\Delta u=f,$ $f\in L^2,$ is in general only in $H^{3/2}$ if $\Omega$ is a nonconvex 
polygonial
domain,
 cf. Grisvard \cite{Grisvard Singularities in B}. For more details on these approximation theorems and regularity results we
refer to Amrouche-Bernardi-Dauge-Girault \cite{Amrouche-Bernardi...}, Belgacem-Bernardi-Costabel-Dauge \cite{Costabel Belgacem Bernardi Dauge 1997}, Ciarlet-Hazard-Lohrengel \cite{Ciarlet et alt. 1998}, Costabel \cite{Costabel 1990}, Costabel-Dauge \cite{Costabel Dauge 1998} and Girault-Raviart \cite{Girault-Raviart}. For a different approach in proving the classical Gaffney inequality for nonsmooth domains see Mitrea \cite{Mitrea Gaffney-Friedrichs ineq}, where the inequality is obtained using existence and regularity of an elliptic boundary value system established in Mitrea \cite{Mitrea Mitrea uses for proof of Gaffney-Friedrichs}.

Note that the proof of Theorem \ref{theorem:lambda wedge omega for vector fields} (Gaffney inequality with condition \eqref{eq:intro:lambda times main condition}) would not simplify if we assumed $\Omega$ to be smooth.

In view of the condition \eqref{eq:intro:lambda times main condition} one might expect that the classical condition $\langle\nu;\omega\rangle=0$ can be generalized too, by replacing $\nu$ by a nonvanishing vector field $\lambda.$ This is however not true if $n\geq 3$ as can be seen by a simple counterexample. It is also not true that condition \eqref{eq:intro:lambda times main condition} generalizes to differential forms of  higher order. We give these counterexamples at the end of this paper.

\section{Mixed classical boundary conditions}

If $\Omega$ is a bounded $C^{1,1}$ open set with unit exterior normal $\nu$ on its boundary $\delomega$ and $\omega$ is some vector field, we shall decompose it as
$$
 \omega=\omega_T+\omega_N,\quad\text{where}\quad \omega_N=\langle\omega;\nu\rangle\nu\quad\text{ and }\quad\omega_T=\omega-\omega_N.
$$
Throughout this paper for vectors fields $\omega,\lambda$ in $\re^n,$ the curl and cross product are defined as vectors in $\re^{\binom{n}{2}}$ defined by
$$
 (\curl \omega)_{ij}=\frac{\partial\omega_j}{\partial x_i}-\frac{\partial\omega_i}{\partial x_j}\quad\text{ and }\quad
 (\omega\times\lambda)_{ij}=\omega_i\lambda_j-\omega_j\lambda_i\,,
 \quad 1\leq i<j\leq n.
$$
We now state a theorem whose proof is essentially the same as the one presented in Csat\'o-Dacorogna-Kneuss \cite{Csato-Dacorogna-Kneuss} for the classical Gaffney inequality.

\begin{theorem}
\label{theorem:mixed classical bdry conditions}
Let $n\geq 2$ and $\Omega\subset\re^n$ be a bounded open $C^{1,1}$ set with exterior unit normal $\nu$ on $\delomega.$ Then there exists a constant $C=C(\Omega)$ such that
$$
  \|\nabla\omega\|_{L^2(\Omega)}^2\leq C\left(\|\curl\omega\|_{L^2(\Omega)}^2+ \|\diw\omega\|_{L^2(\Omega)}^2+\|\omega\|_{L^2(\Omega)}^2\right),
$$
for all $\omega\in H^1(\Omega)^n$ satisfying
$$
  \omega_T=0\text{ or }\omega_N=0\quad\text{ on }\Gamma_i,\qquad \delomega= \bigcup_{i=1}^M\overline{\Gamma}_i\,,
$$
and $\Gamma_i$ are open sets in $\delomega$ and $M\in \mathbb{N}.$
\end{theorem}

\begin{remark}
\label{remark:von Wahl and topology}
If $M=1$ (classical boundary conditions,  $\Gamma_1=\delomega$) and $\Omega$ is contractible, then one easily obtains the better estimate
$$
  \|\nabla\omega\|_{L^2(\Omega)}^2\leq C\left(\|\curl\omega\|_{L^2(\Omega)}^2+ \|\diw\omega\|_{L^2(\Omega)}^2\right),
$$
see Csat\'o-Dacorogna-Kneuss \cite{Csato-Dacorogna-Kneuss} Theorem 6.5 and Theorem 6.7 (Step 1 of the proof). A precise treatment of the optimal topological assumptions on the domain for such an estimate to hold true is carried out in von Wahl \cite{Wahl}.
\end{remark}

\begin{proof} We will not give a detailed proof.
The result follows from \cite{Csato-Dacorogna-Kneuss} Theorem 5.7 (see also \cite{Grisvard Elliptic Pr in nonsmooth} Theorem 3.1.1.1) in the same way as the classical Gaffney inequality: Indeed as in the proof of Theorem 5.16 in \cite{Csato-Dacorogna-Kneuss}  one obtains that
$$
  \intomega(\left(|\curl\omega|^2+|\diw\omega|^2\right)\geq\intomega|\nabla\omega|^2-C\intdelomega|\omega|^2,
$$
and one concludes similarly. The above mentioned references treat $C^2$ domains but remain valid without any change for $C^{1,1}$ domains.
See also Step 3 in the first proof of Proposition \ref{proposition:lambda wedge omega for vector fields}.
\end{proof}

\section{The $\lambda\times\omega=0$ condition}

We now state our first main result. We will distinguish the case $n=2$ as we will give in Section \ref{section n is 2} in the two dimensional case an improvement of the theorem by weakening the regularity assumptions . To state the theorem we need the following definition (which we will use actually only for $C^{r,\alpha}=C^{1,0}$ or $C^{0,1}$).

\begin{definition}
\label{definition:C r alpha on Lipschitz}
Let $r\geq 0$ be an integer and $0\leq\alpha\leq 1.$ If $\Omega$ is a Lipschitz set (meaning that $\delomega$ is Lipschitz), we say that a function $\lambda:\delomega\to\re$ is in $C^{r,\alpha}(\delomega)$ if there exists an extension of $\lambda$ to $\re^n$ such that $\lambda\in C^{r,\alpha}(\re^n).$ We make the convention that $C^{r,0}=C^r.$

\end{definition}

\begin{theorem}\label{theorem:lambda wedge omega for vector fields}
Let $n\geq 2,$ $\Omega\subset\re^n$ be a bounded open Lipschitz set and $\lambda\in C^{1}(\delomega)^n$ be such that
$$
  \lambda\neq 0\quad\text{ on }\delomega.
$$
Then there exists a constant $C=C(\Omega,\lambda)$ such that
$$
  \|\nabla\omega\|_{L^2(\Omega)}^2\leq C\left(\|\curl\omega\|_{L^2(\Omega)}^2+ \|\diw\omega\|_{L^2(\Omega)}^2+\|\omega\|_{L^2(\Omega)}^2\right),
$$
for all $\omega\in H^1(\Omega)^n$ which satisfy
$$
  \lambda\times\omega=0\quad\text{ on }\delomega.
$$
If $n=2$ then the same conclusion holds under the weaker regularity assumptions $\lambda\in C^{0,1}(\delomega)^2.$
\end{theorem}

\begin{remark} \label{remark:main.theorem.lambda}
(i) Note that if $\Omega$ is a $C^{2}$ set, then the unit outward normal vector $\nu$ is $C^{1}$ and the Theorem implies the classical boundary condition $\nu\times\omega=0.$
\smallskip

(ii) If $n=2,$ then this theorem interpolates between the two classical boundary conditions $\omega_T=0,$ respectively $\omega_N=0$. To see this take
$\lambda=\nu=(\nu_1,\nu_2),$ respectively $ \lambda=(\nu_2,-\nu_1).$
\smallskip

(iii)
Recall (see Remark \ref{remark:von Wahl and topology}) that if $\Omega$ is contractible and $\lambda=\nu,$ then in the above theorem the inequality can be replaced by
$$
  \|\nabla\omega\|_{L^2(\Omega)}^2\leq C\left(\|\curl\omega\|_{L^2(\Omega)}^2+ \|\diw\omega\|_{L^2(\Omega)}^2\right),
$$
This is not true for general $\lambda.$ To see this just notice that one can take $\lambda\in C^{\infty}(\Omegabar)^n$ equal to a  harmonic field (i.e. $\curl \lambda=0$ and $\diw \lambda=0$) that never vanishes on the boundary. Then $\omega=\lambda$ trivially satisfies $\lambda\times\omega=0$ on $\delomega.$ Such non-constant harmonic fields exist, for example take $\omega=(x_2,x_1)$ and a domain $\Omega\subset\re^2$ such that $0\notin\delomega,$ so that $\lambda=\omega\neq 0$ on the boundary.

\smallskip
(iv) If $\lambda$ is constant then $C=1,$ see Lemma \ref{lemma:C is one if lambda constant} or proof of Proposition \ref{proposition:lambda wedge omega for vector fields}. $C=1$ also if $\lambda=\nu$ is the normal to $\delomega$ and $\Omega$ is convex (actually $n-1$ convex is sufficient), see \cite{Csato-Dacorogna-Sil}, but this requires a different proof.
\end{remark}

\begin{proof}[Proof of Theorem \ref{theorem:lambda wedge omega for vector fields}] We first prove the result for $C^1$ vector fields $\omega$, respectively Lipschitz vector fields if $n=2$  (cf. Proposition \ref{proposition:lambda wedge omega for vector fields}).
Theorem \ref{theorem:lambda wedge omega for vector fields} will then follow by approximation (cf. Proposition \ref{prop:approx of lambda times omega zero by smooth}).
\end{proof}

\begin{proposition}\label{proposition:lambda wedge omega for vector fields}
Let $n\geq 2,$ $\Omega\subset\re^n$ be a bounded open Lipschitz set and $\lambda\in C^{0,1}(\delomega)^n$ be such that $\lambda\neq 0$ on $\delomega.$
Then there exists a constant $C=C(\Omega,\lambda)$ such that
$$
  \|\nabla\omega\|_{L^2(\Omega)}^2\leq C\left(\|\curl\omega\|_{L^2(\Omega)}^2+ \|\diw\omega\|_{L^2(\Omega)}^2+\|\omega\|_{L^2(\Omega)}^2\right),
$$
for all $\omega\in C^{1}(\Omegabar)^n$ which satisfy $\lambda\times \omega=0$ on $\delomega.$ If $n=2,$ the the same holds true if $\omega\in C^{0,1}(\Omegabar)^2$.
\end{proposition}

\begin{remark}
Note that in this proposition we require that $\lambda$ is only Lipschitz. The loss of regularity compared to the main Theorem \ref{theorem:lambda wedge omega for vector fields} arises in the approximation, see Proposition \ref{prop:approx of lambda times omega zero by smooth}.
\end{remark}

We give two proofs of this proposition. The first one is simpler, following the ideas of Csat\'o-Dacorogna \cite{Csato-Dacorogna 2010}. However, we do not use the identity established in \cite{Csato-Dacorogna 2010} and which is used in establishing the classical Gaffney inequality, respectively Theorem \ref{theorem:mixed classical bdry conditions}. The second proof that we give is a generalization of Morrey's original proof of Gaffney inequality (see Morrey \cite{Morrey 1956}, Morrey-Eells \cite{MorreyEells} or Iwaniec-Scott-Stroffolini \cite{Iwaniec Scott} for an $L^p$ version) for the boundary condition $\nu\times\omega=0.$ It is longer, but several of the intermediate steps  are of interest on their own right, cf. Lemma \ref{lemma:C is one if lambda constant}, and also Lemmas \ref{lemma:d delta nabla invariant in squar if SOn} and \ref{lemma:key lemma local boundary for lambda wedge 1form} which are independent of the boundary conditions. In the first proof we will use the following abbreviation, $f$
being a function defined on a neighborhood of $\delomega$:
$$
  \p_{ij}[f]:=\nu_j\frac{\p f}{\p x_i}-\nu_i\frac{\p f}{\p x_j},
$$
where $\nu=(\nu_1,\ldots,\nu_n)$ is the outward unit normal vector on $\delomega.$ It can be easily seen that $\p_{ij}[f]$ is a tangential derivative and depends only on the values of $f$ on $\delomega.$ 
Therefore, if $f$ is Lipschitz then $\partial_{ij}[f]$ is well defined $\mathcal{H}^{n-1}$ almost everywhere on any Lipschitz boundary $\delomega$, see for instance Lemma \ref{lemma:proof of identity for pair of lipscitz vectors} equations (10)-(11) for the case $n=2$ (if $n\geq 3,$ the argument is similar by  composing $f$ with a local parametrization of $\delomega$). Moreover by the product rule of derivation:
\begin{equation}
 \label{eq:product rule for p ij tangential der}
  \p_{ij}[fg]=\p_{ij}[f] g+f\p_{ij}[g].
\end{equation}
Throughout the proof we will frequently use that any Lipschitz function defined on a subset of $\re^n$ can be extended to a Lipschitz function on the whole space, and conversely, that  restrictions of Lipschitz functions to any subset are still Lipschitz.

\smallskip

\begin{proof}[First Proof of Proposition \ref{proposition:lambda wedge omega for vector fields}.]
\textit{Step 1.}
Let us assume first that $\omega\in C^2(\Omegabar)^n.$ A direct calculation gives the identity
$$
  |\curl\omega|^2+|\diw\omega|^2-|\nabla\omega|^2=2\sum_{i<j}\left(\frac{\partial\omega_i}{\partial x_i}\frac{\partial\omega_j}{\partial x_j}-\frac{\partial\omega_i}{\partial x_j}\frac{\partial\omega_j}{\partial x_i}\right).
$$
So we obtain by partial integration that
\begin{equation}
 \label{eq:proof:first proof lambda wedge partial int}
  \intomega\left(|\curl \omega|^2+|\diw \omega|^2-|\nabla\omega|^2\right)=
  -\sum_{i<j}
  \intdelomega \omega_i\p_{ij}[\omega_j]+\sum_{i<j}\intdelomega \omega_j \p_{ij}[\omega_i].
\end{equation}
Note that \eqref{eq:proof:first proof lambda wedge partial int} involves only first derivatives of $\omega.$
Therefore by approximation one directly deduces 
that \eqref{eq:proof:first proof lambda wedge partial int} remains true for any $\omega\in C^{1}(\Omegabar)^n.$ To see this note that standard convolution in the whole space works, since the derivatives of $\omega$ are uniformly continuous, and the derivatives of the approximating sequence will converge also uniformly on $\delomega.$ If $n=2$ we apply Lemma \ref{lemma:proof of identity for pair of lipscitz vectors} to obtain that \eqref{eq:proof:first proof lambda wedge partial int} remains true if $\omega\in C^{0,1}(\Omegabar)^2$.
\smallskip

\textit{Step 2.} Since $\lambda=(\lambda_1,\ldots,\lambda_n)\neq 0$ on $\delomega$,
there exist open sets $W_1,\cdots,W_M$, integers $1\leq k(1),\cdots, k(M)\leq n$ and $\epsilon>0$ such that
$$
  \delomega\subset \bigcup_{l=1}^MW_l\quad \text{and}\quad |\lambda_{k(l)}|\geq \epsilon\text{ in }W_l\text{ for $1\leq l\leq M.$}
$$
 We now define inductively
$$
  S_1=W_1\cap \delomega,\quad {S}_2=(W_2\cap \delomega)\backslash {S}_1\,,\ldots,\quad {S}_j=(W_j\cap\delomega)\backslash\left(\bigcup_{m=1}^{j-1}{S}_m\right),
$$
for $j=1,\ldots,M.$ Thus the ${S}_j$ form a disjoint union of $\delomega$ and we can write
\begin{equation}
 \label{eq:proof:p ij omega i omega j split to Gamma l}
  \intdelomega\left(-\omega_i\p_{ij}[\omega_j]+\omega_j\p_{ij}[\omega_i]\right)
  =\sum_{l=1}^M\int_{{S}_l}\left(-\omega_i\p_{ij}[\omega_j]+\omega_j\p_{ij}[\omega_i]\right),
\end{equation}
for any $i<j.$ We now claim that for each $l=1,\ldots,M$ and each $i<j,$ there exists a constant  $C=C(\Omega,\lambda)>0$ such that
\begin{equation}
 \label{eq:proof:p ij omega j- p ij omega i and so on}
  \left|\int_{{S}_l}\left(-\omega_i\p_{ij}[\omega_j]+\omega_j\p_{ij}[\omega_i]\right) \right|\leq C\int_{{S}_l}|\omega|^2
\end{equation}
for any $\omega$ satisfying $\lambda\times\omega=0$ on $\delomega.$ Indeed fix $l$ and assume without loss of generality that $k(l)=1.$ Then we obtain from the boundary condition on $\omega$ that
$\lambda_1\omega_i-\lambda_i\omega_1=0$ for $i=1,\ldots,n$
on $\delomega.$
Thus we first obtain that for $i=1,\ldots,n$
$$
  \omega_i=\mu_i\omega_1\quad\text{ and where }\quad \mu_i =\frac{\lambda_i}{\lambda_1}\in C^{0,1}(\overline{{S}_i}).
$$
This gives, using \eqref{eq:product rule for p ij tangential der}, that on ${S}_l$ we have
$$
  \left(-\omega_i\p_{ij}[\omega_j]+\omega_j\p_{ij}[\omega_i]\right)=-\omega_1^2\left(\mu_i\p_{ij}[\mu_j]-\mu_j\p_{ij}[\mu_i]\right).
$$
From this indentity we obtain \eqref{eq:proof:p ij omega j- p ij omega i and so on}.
\smallskip

\textit{Step 3.} From \eqref{eq:proof:first proof lambda wedge partial int}, \eqref{eq:proof:p ij omega i omega j split to Gamma l} and \eqref{eq:proof:p ij omega j- p ij omega i and so on} it follows that
$$
  \intomega\left(|\curl\omega|^2+|\diw\omega|^2-|\nabla\omega|^2\right)\geq -C_1\intdelomega|\omega|^2
$$
for some constant $C_1=C_1(\Omega,\lambda)>0.$
We now recall that there exists a constant $C_2=C_2(\Omega)$ such that (see for instance \cite{Grisvard Elliptic Pr in nonsmooth} Theorem 1.5.1.10 or \cite{Csato-Dacorogna-Kneuss} Proposition 5.15) for any $0<\epsilon<1$
$$
  \intdelomega|\omega|^2\leq \epsilon\intomega|\nabla\omega|^2+\frac{C_2}{\epsilon}\intomega|\omega|^2.
$$
Choose $\epsilon$ such that $\epsilon C_1\leq 1/2$ and then the theorem follows.
\end{proof}
\smallskip

We have used in the proof of Proposition \ref{proposition:lambda wedge omega for vector fields}, in the case $n=2$, the following lemma. In this case one cannot prove \eqref{eq:proof:first proof lambda wedge partial int} for Lipschitz vectors by approximation, since standard convolution by some smoothing kernels $\{\eta_k\}_{k\in\mathbb{N}}$ in the whole space does not imply any kind of convergence of $\left\{\eta_k\ast \partial\omega_i/\partial x_j\right\}_{k\in\mathbb{N}}$ on $\delomega$ to the required function.

\begin{lemma}
\label{lemma:proof of identity for pair of lipscitz vectors}
Let $\Omega\subset\re^2$ be a bounded open Lipschitz set with unit outward normal $\nu$ and assume that $\omega_1\,,\omega_2\in W^{1,\infty}(\Omega).$ Then the following identity holds
\begin{equation}
\label{eq:lemma:pair of lipschitz vectors}
  \intdelomega \omega_1\left(\frac{\p \omega_2}{\p x_2}\nu_1-\frac{\p \omega_2}{\p  
  x_1}\nu_2\right)
  =\intomega
  \left(\frac{\p \omega_1}{\p x_1}\frac{\p \omega_2}{\p x_2}-\frac{\p \omega_2}{\p  
  x_1}\frac{\p \omega_1}{\p x_2}\right).
\end{equation}
\end{lemma}

 \begin{proof}
\textit{Step 1.} Clearly \eqref{eq:lemma:pair of lipschitz vectors} holds true for $(\omega_1,\omega_2)\in C^2(\Omegabar)^2,$ by partial integration. Let us first show that \eqref{eq:lemma:pair of lipschitz vectors} holds true if $\omega_1\in C^2(\Omegabar)$ and $\omega_2$ is Lipschitz. Let us first assume that $\delomega$ is connected and hence there exists a Lipschitz curve $\varphi$ and some interval $[0,L]$ such that
\begin{equation}
 \label{eq:varphi param and periodic}
   \varphi:[0,L]\to \delomega,\qquad \varphi(0)=\varphi(L)
\end{equation}
is a parametrization of $\delomega.$ We obtain that $\omega_2\circ\varphi\in W^{1,\infty}([0,L]),$ as it is the composition of two Lipschitz functions, and it is differentiable almost everywhere in $[0,L]$ with
\begin{equation}
 \label{eq:omega circ phi deriv}
  \frac{d}{dt}(\omega_2\circ\varphi)(t)=\frac{\p \omega_2}{\p x_1}(\varphi(t))\varphi_1'(t)
  +\frac{\p \omega_2}{\p x_2}(\varphi(t))\varphi_2'(t)=
  \left(\frac{\p \omega_2}{\p x_1}\nu_1
  -\frac{\p \omega_2}{\p x_2}\nu_2\right)(\varphi(t))|\varphi'(t)|.
\end{equation}
We have assumed here that $\varphi$ turns around the domain counterclockwise.
Thus we obtain, using that $\varphi(0)=\varphi(L),$ $\omega_1\in C^2(\Omegabar)$  (and hence its second derivatives commute)
\begin{align*}
  \intdelomega \omega_1\left(\frac{\p \omega_2}{\p x_2}\nu_1-\frac{\p \omega_2}{\p  
  x_1}\nu_2\right)
  =&\int_0^L\omega_1(\varphi(t))\frac{d}{dt}\left[\omega_2(\varphi(t))\right]\,dt =
  -\int_0^L\frac{d}{dt}\left[\omega_1(\varphi(t))\right]\omega_2(\varphi(t))\,dt
  \smallskip \\
  =&
   -\intdelomega \omega_2\left(\frac{\p \omega_1}{\p x_2}\nu_1-\frac{\p \omega_1}{\p  
  x_1}\nu_2\right)=
  \intomega
  \left(\frac{\p \omega_1}{\p x_1}\frac{\p \omega_2}{\p x_2}-\frac{\p \omega_2}{\p  
  x_1}\frac{\p \omega_1}{\p x_2}\right).
\end{align*}
This proves the claim of the present step, in case $\delomega$ is connected. If $\delomega$ is not connected then we first show that on each connected component $S_i$ of $\delomega$ ($i=1,\ldots, K$ for some $K\in\mathbb{N}$)
$$
  \int_{S_i} \omega_1\left(\frac{\p \omega_2}{\p x_2}\nu_1-\frac{\p \omega_2}{\p  
  x_1}\nu_2\right)
 =
   -\int_{S_i} \omega_2\left(\frac{\p \omega_1}{\p x_2}\nu_1-\frac{\p \omega_1}{\p  
  x_1}\nu_2\right), 
$$
as before, taking periodic paramterizations $\varphi_i$ of $S_i$. Then we take the sum over these integrals and can proceed in the same way. This proves the claim of Step 1.
\smallskip

\textit{Step 2.} Let us now assume that $\omega_1\,,\omega_2$ are both Lipschitz. Take a sequence $\{\omega_1^k\}\in C^{\infty}(\Omegabar),$ $k\in \mathbb{N},$ such that
$$
   \omega_1^k\to\omega\quad\text{ in }W^{1,2}(\Omega)\quad\text{ for $k\to\infty$.}
$$
By Step 1 we have for each $k$
$$
     \intdelomega \omega_1^k\left(\frac{\p \omega_2}{\p x_2}\nu_1-\frac{\p \omega_2}{\p  
  x_1}\nu_2\right)
  =\intomega
  \left(\frac{\p \omega_1^k}{\p x_1}\frac{\p \omega_2}{\p x_2}-\frac{\p \omega_2}{\p  
  x_1}\frac{\p \omega_1^k}{\p x_2}\right).
$$
By the trace theorem $\omega_1^k\to\omega_1$ in $L^2(\delomega),$ and by \eqref{eq:omega circ phi deriv}
$$
   \left(\frac{\p \omega_2}{\p x_2}\nu_1-\frac{\p \omega_2}{\p  
  x_1}\nu_2\right)\in L^{\infty}(\delomega).
$$
So by letting $k\to\infty$ we obtain \eqref{eq:lemma:pair of lipschitz vectors}.
\end{proof}

\smallskip

We split the second proof of Propostion \ref{proposition:lambda wedge omega for vector fields} (which requires $\lambda$ to be $C^{1,1}$) into several intermediate steps. We first recall the definition of the pushforward of a vector field.

\begin{definition}
\label{def:pullback of vector field}
Let $U,V\subset\re^n$ be two open sets and $\Phi\in \operatorname{Diff}^1(U;V).$ Then for any $\omega\in C(U)^n$ we define its pushforward $\Phi_{\ast}(\omega)\in C(V)^n$ by
$$
  \Phi_{\ast}(\omega)(x)=\nabla \Phi\left(\Phi^{-1}(x)\right)\,\omega\left(\Phi^{-1}(x)\right),
$$
where $A\,b$ is the usual multiplication of a (column) vector $b$ by a matrix $A.$
\end{definition}

We will use several times the following elementary properties: $(\Phi\circ \Psi)_{\ast}(\omega)=\Phi_{\ast}(\Psi_{\ast}(\omega))$ and
\begin{equation}
  \label{eq:phi ast a times b stays zero}
   \alpha\times\beta=0\quad\text{ at }x\quad\Leftrightarrow\quad \Phi_{\ast}(\alpha)\times\Phi_{\ast}(\beta)=0\quad\text{ at }\Phi(x).
\end{equation}

The proof of the next lemma is a straightforward algebraic calculation. The analoguous result for the pullback of general $k$-forms can be found in \cite{Csato Thesis}, Lemma B.13. However, in the present case of vector fields, the proof is much simpler. $O(n)$ shall denote the set of  orthogonal matrices.

\begin{lemma}
\label{lemma:d delta nabla invariant in squar if SOn}
Let $U,V\subset\mathbb{R}^n$ be open sets, $A\in O(n),$ $b\in\re^n,$ and $\psi:U\to V=\psi(U)$ defined by $\psi(u)=Au+b.$ Then for all $\omega\in C^{0,1}(U)^n$ and almost every $u\in U$ the following three identities hold true:
\begin{equation}
 \label{eq:nabla square invariant under rotation-pullback}
  |\nabla\omega(u)|^2= |\nabla(\psi_{\ast}(\omega))|^2(\psi(u))
\end{equation}
\begin{equation}
 \label{eq:d square invariant under rotation-pullback}
  |\curl\omega(u)|^2= |\curl(\psi_{\ast}(\omega))|^2(\psi(u))
\end{equation}
\begin{equation}
 \label{eq:delta square invariant under rotation-pullback}
  |\diw\omega(u)|^2= |\diw(\psi_{\ast}(\omega))|^2(\psi(u)).
\end{equation}
\end{lemma}

\begin{remark}
This lemma holds true by the specific algebraic properties of $\nabla,$ $\curl$ and $\diw$ and is not valid in general for an arbitrary linear combination of derivatives of $\omega.$ In case of $\diw$ we have actually something stronger: $\diw\omega(u)=\diw(\psi_{\ast}(\omega))(\psi(u))$ for any invertible matrix $A.$
\end{remark}

\begin{proof}
We first prove \eqref{eq:nabla square invariant under rotation-pullback}.
Let $a_{ij}$ denote the entries of the matrix $A.$ Since $A^t=A^{-1}$ we have that for any $k,l=1,\ldots, n,$
\begin{equation}
 \label{eq:sum aik is zero or 1}
  \sum_{i=1}^na_{ik}a_{il}=\delta_{kl}.
\end{equation}
We can assume that $b=0.$ Let $x=\psi(u)=Au,$ and hence $\psi_{\ast}(\omega)(x)=A\omega(A^tx).$ So the components of $\psi_{\ast}(\omega),$ respectively their derivatives are
$$
  \left(\psi_{\ast}(\omega)\right)_i(x)=\sum_{k=1}^na_{ik}\omega_k(A^tx)
  \quad\text{ and }\quad
  \frac{\partial \left(\psi_{\ast}(\omega)\right)_i}{\partial x_j}(x)=\sum_{k,l=1}^na_{ik}a_{jl}\frac{\partial\omega_k}{\partial u_l}(A^tx).
$$
We therefore obtain
\begin{align*}
  |\nabla\psi_{\ast}(\omega)|^2(x)=\sum_{i,j=1}^n\left(\sum_{k,l=1}^na_{ik}a_{jl}\frac{\partial\omega_k}{\partial u_l}(u)\right)^2
  =\sum_{i,j=1}^n\sum_{k,l=1}^n\sum_{r,s=1}^na_{ik}a_{jl}a_{ir}a_{js}\frac{ \partial \omega_k}{\partial u_l}(u)\frac{\partial \omega_r}{\partial u_s}(u).
\end{align*}
Using now \eqref{eq:sum aik is zero or 1} gives the desired result.
To prove \eqref{eq:d square invariant under rotation-pullback} we use that
\begin{align*}
  |\curl(\psi_{\ast}(\omega))|^2=&\sum_{i<j}\left(\frac{\partial \left(\psi_{\ast}(\omega)\right)_j}{\partial x_i}- \frac{\partial \left(\psi_{\ast}(\omega)\right)_i}{\partial x_j}\right)^2=
  \frac{1}{2}\sum_{i,j=1}^n \left(\frac{\partial \left(\psi_{\ast}(\omega)\right)_j}{\partial x_i}- \frac{\partial \left(\psi_{\ast}(\omega)\right)_i}{\partial x_j}\right)^2
  \smallskip \\
  =&\frac{1}{2}\sum_{i,j=1}^n \left( \sum_{k,l=1}^na_{ik}a_{jl}\left(\frac{\partial\omega_k}{\partial u_l}-\frac{\partial \omega_l}{\partial u_k}\right)\right)^2
\end{align*}
and proceed as in the proof of \eqref{eq:nabla square invariant under rotation-pullback}. The proof of \eqref{eq:delta square invariant under rotation-pullback} is very similar.
\end{proof}
\smallskip

We start proving Proposition \ref{proposition:lambda wedge omega for vector fields} in a special case.

\begin{lemma}
\label{lemma:C is one if lambda constant}
Let $\Omega\subset\re^n$ be a bounded open Lipschitz set and let $\lambda\in \re^n$ be a nonzero constant vector. Then the equality
$$
  \intomega|\nabla\omega|^2=\intomega\left(|\curl \omega|^2+|\diw \omega|^2\right)
$$
holds true for all $\omega\in \mathbf{C^{1}}(\Omegabar)^n$ which satisfy
$\lambda\times\omega=0\quad\text{ on }\delomega.$
\end{lemma}

\begin{remark}
We will only use this lemma for $\lambda=e_1$, and will therefore only prove that case. The result for general $\lambda$ follows easily from this particular case, Lemma \ref{lemma:d delta nabla invariant in squar if SOn} and \eqref{eq:phi ast a times b stays zero}.
\end{remark}

\begin{proof}
As remarked, we only prove the lemma in tha case when $\lambda=e_1=(1,0,\ldots,0).$ In this case the boundary condition $\lambda\times\omega=0$ is equivalent with $\omega_2=\cdots=\omega_n=0$ on $\delomega.$
Recall that, see \eqref{eq:proof:first proof lambda wedge partial int},
$$  \intomega\left(|\curl \omega|^2+|\diw \omega|^2-|\nabla\omega|^2\right)=
  -\sum_{i<j}
  \intdelomega \omega_i\p_{ij}[\omega_j]+\sum_{i<j}\intdelomega \omega_j \p_{ij}[\omega_i].
  $$
The right hand side of the previous equality cancels since, for any $i\neq j,$ $\omega_i\p_{ij}[\omega_j]$ is pointwise zero on $\delomega:$ indeed either $\omega_i=0$ on $\delomega$ or
$$
  \p_{ij}[\omega_j]=\frac{\partial\omega_j}{\partial x_i}\nu_j-\frac{\partial \omega_j}{\partial x_j}\nu_i=0\quad\text{ on }\delomega
$$
if $\omega_j=0$ on $\delomega$ recalling that $\p_{ij}[\omega_j]$ is a tangential derivative.
\end{proof}
\smallskip

The main statement of the next Lemma (Part (ii)), states that that the change of the $L^2$ norms of $\nabla\omega,$ $\curl\omega$ and $\diw\omega$ under the pushforward of $\Phi$ can be estimated appropriately, if $\nabla \phi\in SO(n)$ at some point and if a neighborhood is taken small enough near that point.

\begin{lemma}
\label{lemma:key lemma local boundary for lambda wedge 1form}
Let $x_0\in\re^n$ and  $\lambda$ be a $C^{1,1}$ vector field defined in a neighborhood of $x_0$, such that $|\lambda(x_0)|= 1.$
\smallskip

(i)
Then there exist open sets $O,W\subset\re^n,$ $x_0\in O,$ $0\in W,$  and a diffeomorphism $\Phi\in \operatorname{Diff}^{1,1}(\overline{O};\overline{W})$ such that $\Phi(x_0)=0$,
$$
  \Phi_{\ast}(\lambda)=e_1\quad\text{ in }W
  \quad\text{ and }\quad
  \nabla \Phi(x_0)\in SO(n).
$$

(ii)
Moreover for any $0<\epsilon\leq 1,$ up to taking $O$ and $W$ smaller, there exists a constant $C=C(\Phi)$ satisfying the following three inequalities:
\begin{equation}
   \label{eq:Difference between Norms of Gradients}
    \left|\int_O|\nabla\omega|^2-\int_W|\nabla(\Phi_{\ast}(\omega))|^2\right|\leq
    \epsilon\int_O|\nabla\omega|^2+\frac{C}{\epsilon}\int_O|\omega|^2
\end{equation}
\begin{equation}
  \label{eq:difference between d and d-ast}
    \left|\int_O|\curl\omega|^2-\int_W|\curl(\Phi_{\ast}(\omega))|^2\right|\leq
    \epsilon\int_O|\nabla\omega|^2+\frac{C}{\epsilon}\int_O|\omega|^2
\end{equation}
\begin{equation}
  \label{eq:difference between delta and delta-ast}
    \left|\int_O|\diw\omega|^2-\int_W|\diw(\Phi_{\ast}(\omega))|^2\right|\leq
    \epsilon\int_O|\nabla\omega|^2+\frac{C}{\epsilon}\int_O|\omega|^2
\end{equation}
for all $\omega\in C^{0,1}(\overline{O})^n.$
\end{lemma}
\begin{remark}\label{remark:apres epsilon Phi} The proof will actually show that \eqref{eq:Difference between Norms of Gradients}-\eqref{eq:difference between delta and delta-ast} remain valid with the same constant $C$ replacing $O$ by any of its own open subsets $V$ and replacing $W$ by $U=\Phi(V).$
\end{remark}
\begin{proof} Without loss of generality we can assume that $x_0=0.$ \smallskip

\textit{Step 1.} We first prove (i).  Let $\tilde{\Psi}(t,x)$ be the solution of
$$
  \frac{\partial \tilde{\Psi}}{\partial t}=\lambda(\tilde{\Psi})\quad\text{ and }\quad \tilde{\Psi}(0,x)=Ax,
$$
where $A\in SO(n)$ is such that its first column is equal to $\lambda(x_0).$ Then define $\Psi(x)=\tilde{\Psi}(x_1,0,x_2,\cdots,x_n).$ It can be easily verified that $\Phi=\Psi^{-1}$ has all the desired properties.

\textit{Step 2.} We now prove (ii).
We will only do the proof for \eqref{eq:Difference between Norms of Gradients}. The proof for \eqref{eq:difference between d and d-ast} and \eqref{eq:difference between delta and delta-ast} is very similar.
Let $\Phi\in \operatorname{Diff}^{1,1}\left(\overline{O},\overline{W}\right)$ be as in (i) and $\Psi=\Phi^{-1}.$ Throughout the proof $C_1,C_2,C_3$ and $C_4$ will denote constants depending only on $\Phi.$
Let us write
\begin{align*}
  \nabla\left(\Phi_{\ast}(\omega)\right)(y)=&\nabla\left((\nabla \Phi\circ\Psi) (\omega\circ\Psi)\right)(y)
  \smallskip \\
  =& \sum_{k=1}^n S^k(\Phi,y)\omega_k(\Psi(y))+\nabla \Phi(\Psi(y))\,\nabla\omega(\Psi(y))\,\nabla \Psi(y),
\end{align*}
where $S^k(\Phi,y),$ $k=1,\ldots,n,$ are matrix valued functions depending only on derivatives of at most second order of $\Phi.$ Its entries shall be donoted by $S^k_{ij}(\Phi,y).$ So we have
\begin{equation}\label{eq:decomp.nabla.phi.ast}
  \left|\nabla\left(\Phi_{\ast}(\omega)\right)\right|^2=D+E+F,
\end{equation}
where
\begin{align*}
  D(y)=&\sum_{i,j=1}^n\Big(\nabla \Phi(\Psi(y))\,\nabla\omega(\Psi(y))\,\nabla \Psi(y)\Big)_{ij}^2,\qquad F(y)=\sum_{i,j=1}^n\left(\sum_{k=1}^nS_{ij}^k(\Phi,y)\omega_k(\Psi(y))\right)^2
  \smallskip \\
  E(y)=&2\sum_{i,j,k=1}^nS_{ij}^k(\Phi,y)\omega_k(\Psi(y))\left(\nabla \Phi(\Psi(y))\,\nabla\omega(\Psi(y))\,\nabla \Psi(y)\right)_{ij}
\end{align*}
Fix $0<\epsilon\leq 1.$ Using the inequality $2ab\leq a^2/\epsilon+b^2\epsilon$ and the fact that $\Phi$ is $C^{1,1},$ one immediately obtains
$$
  E(y)\leq C_1\epsilon|\nabla\omega|^2(\Psi(y))+\frac{C_1}{\epsilon}|\omega|^2(\Psi(y))\quad \text{and}\quad F(y)\leq C_2|\omega|^2(\Psi(y))\quad
  \quad\text{ for all $y\in \overline{O}.$}
$$
Changing the variables we therefore get
\begin{equation}\label{eq:E}\int_{W}E\leq \int_{O}\left(C_1\epsilon|\nabla\omega|^2(x)+\frac{C_1}{\epsilon}|\omega|^2(x)\right)\det\nabla\Phi(x)dx
\leq\int_{O}\left(C_3\epsilon|\nabla\omega|^2+\frac{C_3}{\epsilon}|\omega|^2\right)
\end{equation}
and similarly
\begin{equation}\label{eq:F}
\int_{W}F\leq\int_{O}C_4|\omega|^2.
\end{equation}
Combining \eqref{eq:decomp.nabla.phi.ast}, \eqref{eq:E} and \eqref{eq:F} it is enough to estimate
$$
\left|\int_WD-\int_O|\nabla\omega|^2\right|
$$
to prove \eqref{eq:Difference between Norms of Gradients}.
By the change of variables formula we get
$$
\int_WD=\int_O|\nabla \Phi(x)\,\nabla\omega(x)\,(\nabla \Phi(x))^{-1}|^2\det\nabla\Phi(x)dx.
$$
Thus
\begin{align*}\int_WD-\int_O|\nabla\omega|^2&=\int_{O}\left(\left|\nabla \Phi(x)\,\nabla\omega(x)\,(\nabla \Phi(x))^{-1}\right|^2\det\nabla\Phi(x)-\left|
   \nabla \Phi(0)\,\nabla\omega(x)\,(\nabla \Phi(0))^{-1}\right|^2\right)dx\\
   &+\int_{O}\left(\left|\nabla \Phi(0)\,\nabla\omega(x)\,(\nabla \Phi(0))^{-1}\right|^2-|\nabla\omega(x)|^2\right)dx.
\end{align*}
It follows from \eqref{eq:nabla square invariant under rotation-pullback} that the integrand in the second integral of the right-hand side of the previous equation is pointwise 0 in $O.$ To see this, fix $x\in O,$ set $A=\nabla \Phi(0)\in SO(n)$ and apply the map $\psi(u)=Au$
to Lemma \ref{lemma:d delta nabla invariant in squar if SOn}: then \eqref{eq:nabla square invariant under rotation-pullback} evaluated at $u=x$ gives
$$
  |\nabla\omega(x)|^2=|\nabla(\psi_{\ast}(\omega))|^2(\psi(x))=\left|A\, \nabla\omega(x)\,A^t\right|^2=\left|\nabla \Phi(0)\,\nabla\omega(x)\,(\nabla \Phi(0))^{-1}\right|^2.
$$
Hence, recalling that $\det\nabla\Phi(0)=1$, it follows from continuity of $\nabla \Phi$ that, taking $O$ smaller (and consequently $W$ as well) if necessary, that
$$
  \left|\int_WD-\int_O|\nabla\omega|^2\right|\leq \epsilon \int_{O}|\nabla\omega|^2.
$$
This concludes the proof of the lemma since the estimates on $E$ and $F$ remain valid for the new smaller open sets $O$ and $W$ with the same constants $C_1,C_2,C_3$ and $C_4.$
\end{proof}

\smallskip

We now prove Proposition \ref{proposition:lambda wedge omega for vector fields} in the special case when the vector fields $\omega$ have compact support in a sufficiently small neighborhood of a boundary point $x_0\in\delomega.$

\begin{lemma}
\label{lemma:theorem on lambda wedge near x0}
Let $\Omega\subset\re^n$ be a bounded open Lipschitz set and $\lambda\in C^{1,1}(\delomega)^n$ be such that $\lambda\neq 0$ on $\delomega$ and assume that $x_0\in\delomega.$ Then there exists an open set $O\in \re^n,$ $x_0\in O$ and a constant $C=C(\Omega,O,\lambda)$ such that
$$
  \int_V|\nabla\omega|^2\leq C\int_V\left(|\curl\omega|^2+|\diw\omega|^2+|\omega|^2\right),
$$
where $V=\Omega\cap O,$ for all $\omega\in \mathbf{C^{1}}\left(\overline{O}\right)$ which satisfy
$$
  \lambda\times\omega=0\quad\text{ on }\delomega\quad\text{ and }\quad \operatorname{supp}(\omega)\subset O.
$$
\end{lemma}

\begin{proof}
The proof follows from Lemmas \ref{lemma:C is one if lambda constant} and \ref{lemma:key lemma local boundary for lambda wedge 1form}. With no loss of generality we can assume that $|\lambda(x_0)|=1.$  We claim that $O$ given by  Lemma \ref{lemma:key lemma local boundary for lambda wedge 1form} will have the desired property and we shall use the notation of that lemma. If $\lambda\times\omega=0$ on $\delomega,$ we get that, using \eqref{eq:phi ast a times b stays zero},
$$
  e_1\times \Phi_{\ast}(\omega)=0\quad\text{ on }\Phi(\partial\Omega\cap O)
$$
Also, since $\omega$ has compact support in $O,$ then $\Phi^{\ast}(\omega)$ has compact support in $\Phi(O).$ We conclude that, setting $U=\Phi(V),$
$$
  e_1\times\Phi_{\ast}(\omega)=0\quad\text{ all over }\partial U
$$
for any $\omega$ satisfying the assumptions of the lemma. We thus conclude from Lemma \ref{lemma:C is one if lambda constant} that
$$
  \int_U|\nabla(\Phi_{\ast}(\omega))|^2=\int_U\left(|\curl(\Phi_{\ast}(\omega))|^2+ |\diw(\Phi_{\ast}(\omega))|^2\right).
$$
Finally using \eqref{eq:Difference between Norms of Gradients}--\eqref{eq:difference between delta and delta-ast} (and Remark \ref{remark:apres epsilon Phi}) with $\epsilon=1/6$ and the previous equality we obtain that
\begin{align*}
  \int_V|\nabla\omega|^2\leq & \epsilon \int_V|\nabla\omega|^2+\frac{C}{\epsilon}\int_V|\omega|^2+\int_U|\nabla(\Phi^{\ast}(\omega))|^2
  \smallskip
  \\
  \leq & 3\epsilon \int_V|\nabla\omega|^2+3\frac{C}{\epsilon}\int_V|\omega|^2+ \int_V\left(|\curl\omega|^2+|\diw\omega|^2\right)
  \smallskip
  \\
  =&\frac{1}{2}\int_V|\nabla\omega|^2+18C\int_V|\omega|^2+\int_V \left(|\curl\omega|^2+|\diw\omega|^2\right).
\end{align*}
which proves the lemma.
\end{proof}
\smallskip

We give the second proof of the main proposition under the more restrictive hypothesis that $\lambda\in C^{1,1}(\partial\Omega)^n.$
\smallskip

\begin{proof}[Second Proof (Proposition \ref{proposition:lambda wedge omega for vector fields}).]
Since $\delomega$ is compact, we can cover it by open neighborhoods $O_i\subset\re^n,$ $i=1,\ldots,M$ which satisfy the conclusion of Lemma \ref{lemma:theorem on lambda wedge near x0}. Moreover let us choose a further open set $O_0\subset \overline{O_0}\subset\Omega$ such that
$
  \Omega\subset\cup_{i=0}^MO_i.
$
Let $\{\xi_i\}_{i=0}^M$ be a partition of unity  subordinate to the $O_i$:
$$
  0\leq\xi_i\leq 1,\qquad\operatorname{supp}(\xi_i)\subset O_i\quad\text{ and }\quad \sum_{i=0}^M\xi_i=1\quad\text{ in $\Omegabar$}.
$$
Let now $\omega\in C^{1}\left(\Omegabar\right)^n$ be a vector field such that $\lambda\times\omega=0$ on $\delomega.$ Then using Lemma
\ref{lemma:C is one if lambda constant} for $i=0$, respectively Lemma \ref{lemma:theorem on lambda wedge near x0} for $i=1,\ldots,M,$ we obtain that
\begin{equation}
 \label{eq:proof:local estimate in Vi}
  \int_{V_i}|\nabla(\xi_i\omega)|^2\leq C_i\int_{V_i}\left(|\curl(\xi_i\omega)|^2 +|\diw(\xi_i\omega)|^2+|\xi_i\omega|^2\right),
\end{equation}
for some constants $C_i=C_i(\Omega,\lambda),$ where $V_i=\Omega\cap O_i$. Note that
\begin{equation}
 \label{eq:proof:local global Vi estimate for nabla}
  \intomega|\nabla\omega|^2= \intomega\left|\nabla\left(\sum_{i=0}^M\xi_i\omega\right) \right|^2 \leq M\sum_{i=0}^M\intomega |\nabla(\xi_i\omega)|^2
  =M\sum_{i=0}^M\int_{V_i}|\nabla(\xi_i\omega)|^2.
\end{equation}
Thus combining \eqref{eq:proof:local global Vi estimate for nabla} and \eqref{eq:proof:local estimate in Vi} one gets
$$
  \intomega|\nabla\omega|^2\leq C_1\sum_{i=0}^M\int_{V_i}\left(|\curl\xi_i\omega|^2+|\diw \xi_i\omega|^2+|\xi\omega|^2\right)
  \leq C_2\intomega \left(|\curl\omega|^2+|\diw\omega|^2+|\omega|^2\right),
$$
for some constants $C_1$ and $C_2$ depending only on $\Omega$ and $\lambda.$
\end{proof}
\smallskip

To extend Proposition \ref{proposition:lambda wedge omega for vector fields} to $H^1$ vector fields we need to show that a vector field $\omega\in H^1(\Omega)^n$ which satisfies $\lambda\times\omega=0$ on the boundary
can be approximated by $C^{1}$ vector fields also satisfying
the same boundary condition. This is possible according to the next proposition.

\begin{proposition}
\label{prop:approx of lambda times omega zero by smooth}
Let $n\geq 2,$ $r\geq 0$ be an integer and $0\leq \alpha\leq 1,$ with $r+\alpha\geq 1.$ Suppose $\Omega\subset\re^n$ is a bounded open Lipschitz set and
$\lambda\in C^{r,\alpha}(\delomega)^n$ be such that
$$
  \lambda\neq 0\quad\text{ on }\delomega.
$$
Suppose $\omega\in H^1(\Omega)^n$ is such that $\lambda\times\omega=0$ on $\delomega.$ Then there exists a 
sequence $\{\omega^k\}_{k\in\mathbb{N}}
\subset C^{r,\alpha}(\Omegabar)^n$ such that for $k\to\infty$
$$
  \omega^k\to\omega\in H^1(\Omega)^n
  \quad\text{ and }\quad
  \lambda\times\omega^k=0\quad\text{ on }\delomega\text{ for all }k.
$$
\end{proposition}

\begin{proof}\textit{Step 1.} We first prove the following claim:
For every $x_0\in\delomega$ there exists a neighborhood $W\subset\re^n$ of $x_0$ such that for all $\omega\in H^1(\Omega)$ satisfying
\begin{equation}
 \label{additional assumption supp omega in W}
  \operatorname{supp}(\omega)\subset W\quad\text{ and }\quad \lambda\times \omega=0\quad\text{ on }\delomega,
\end{equation}
there exists a sequence 
$\{\omega^k\}_{k\in\mathbb{N}}\subset C^{r,\alpha}\left(\overline{\Omega\cap W}\right)^n$ such that
$$
  \omega^k\to \omega\quad\text{ in }H^1\left(\Omega\cap W\right)^n\quad\text{ and }\quad \lambda\times \omega^k=0\quad\text{ on }\delomega\cap W\text{ for all }k.
$$
We extend $\lambda$ to a $C^{r,\alpha}$ vector field in $\re^n,$
see Definition \ref{definition:C r alpha on Lipschitz}.
Since $\lambda$ does not vanish on the boundary we can assume with no loss of generality that
$\lambda_1\neq 0$ in $\overline{W}$ where $W$ is a small enough neighborhood of $x_0.$
Let us define
$$
  \alpha_i=\lambda_1\omega_i-\lambda_i\omega_1=(\lambda\times\omega)_{1i}\,.
$$
Note that by the additional assumptions \eqref{additional assumption supp omega in W}  the support of $\omega$ is contained in  $ W,$ and in particular vanishes on $\partial W.$ Therefore $\alpha_i\in H_0^1(\Omega\cap W)$ and hence there exists a sequence $\alpha_i^{k}$ with the properties
$$
  \{\alpha_i^{k}\}_{k\in \mathbb{N}}\in C^{\infty}_c(\Omega\cap W),\qquad \alpha_i^{k}\to \alpha_i\quad\text{ in }H^1(\Omega\cap W).
$$
Moreover we choose a sequence $\{\beta^{k}\}\in C^{\infty}(\overline{\Omega\cap W})$ such that
$\beta^{k}\to \omega_1$ in $H^1(\Omega\cap W).$
We finally define $\omega^{k}=(\omega_1^{k},\ldots,\omega_n^{k})$ by
\begin{align*}
 \omega_1^{k}=&\beta^{k}\smallskip
 \\
 \omega_i^{k}=&\frac{\alpha_i^{k}+\lambda_i\beta^{k}}{\lambda_1}
                        \quad\text{ for }i=2,\ldots,n.
\end{align*}
Using that $\alpha_i^{k}=0$ on $\delomega\cap W$ we obtain that for any $i,j\in\{1,\ldots,n\}$
$$
  \lambda_j\omega_i^{k}-\lambda_i\omega_j^{k}= \frac{\lambda_j}{\lambda_1}\lambda_i\beta^{k}-\frac{\lambda_i}{\lambda_1}\lambda_j\beta^{k}=0\quad\text{ on }\delomega\cap W\,.
$$
and thus $\omega^k$ has all the desired properties claimed in Step 1.\smallskip

\textit{Step 2.}
Using that $\delomega$ is compact, we can cover it by a finite number of open sets $W_1,\ldots,W_L$ with the properties given by Step 1. Clearly we can add $W_0$ such that $W_0$ is also open, $\Omegabar\subset\bigcup_{l=0}^LW_l$
and any $\omega^0\in H^1(W_0)$ with compact support in $W_0$ can be approximated by smooth vector fields $\omega^{0,k}$ with compact support in $W_0$. In particular $\lambda\times\omega^{0,k}=0$ on $\delomega$ for all $k.$ Let $\eta_l$ be a smooth partition of unity subordinate to this covering such that
$$
  \sum_{l=0}^L\eta_l^2=1\quad\text{ in }\Omegabar.
$$
Define $\omega^l=\eta_l\omega.$ Using Step 1 there exists for
each $l=1,\ldots, L$ sequences $\{\omega^{l,k}\}_{k\in\mathbb{N}}$ of $C^{r,\alpha}$ vector fields such that for $k\to\infty$
$$
  \omega^{l,k}\to \omega^l\quad\text{ in }H^1(\Omega\cap W)\quad\text{ and }\quad \lambda\times \omega^{l,k}=0\quad\text{ on }\delomega\cap W\text{ for all }k.
$$
Then $\eta_l \omega^{l,k}\in C^{r,\alpha}(\overline{\Omega})^n$ is well defined and $\omega^k=\sum_{l=0}^L\eta_l \omega^{l,k}$
has all the desired properties.
\end{proof}

\section{Formulation in $\re^2$ for discontinuous $\lambda$}
\label{section n is 2}

In two dimensions we improve Theorem \ref{theorem:lambda wedge omega for vector fields}: we no longer require $\lambda$ to be continuous on the whole boundary, but still Lipschitz on different pieces of $\delomega.$
More precisely we make the following assumption.

\begin{assumption}\label{Assumption lambda not glob lip}
Assume that $\Omega\subset \re^2$ is a bounded open Lipschitz set, such that for some integer $N$
$$
  \delomega=\bigcup_{i=1}^N\overline{\Gamma}_i\quad\text{ and } \quad
  \overline{\Gamma}_{i}\cap \overline{\Gamma}_{i+1}=\{S_i\}\text{ for $i=1,\ldots,N$},
$$
where $\Gamma_i$ are disjoint open sets in $\delomega$ (with the convention that
$\Gamma_{N+1}=\Gamma_1$) and the $S_i$ are $N$ different points on the boundary, called vertices.
Let $\lambda_i\in C^{0,1}\left(\overline{\Gamma}_i\right)^2$ for $i=1,\ldots,N$ and define
$$
  \lambda:\bigcup_{i=1}^N\Gamma_i\to \re^2,
$$
by $\lambda^i=\lambda$ on $\Gamma_i.$ We also assume that
$$
  \lambda_i\neq 0\quad\text{ on }\overline{\Gamma}_i\,.
$$
\end{assumption}

Note that we allow that at a vertex $S_i$ the segments $\Gamma_i$ and $\Gamma_{i+1}$ can meet at an angle $\pi.$
In this setting we have the following theorem.

\begin{theorem}\label{theorem:lambda wedge omega polygonial in re2}
Let $\Omega$ and $\lambda$ be as in Assumption \ref{Assumption lambda not glob lip}.
Then there exists a constant $C=C(\Omega,\lambda)$ such that
\begin{equation}\label{eq:theorem:main polygonail Gaffney}
  \|\nabla\omega\|_{L^2(\Omega)}^2\leq C\left(\|\curl\omega\|_{L^2(\Omega)}^2+ \|\diw\omega\|_{L^2(\Omega)}^2+\|\omega\|_{L^2(\Omega)}^2\right),
\end{equation}
for all $\omega\in H^1(\Omega)^2$ which satisfy
$$
  \lambda\times\omega=0\quad\text{ on }\delomega,
$$
where the last equality is understood as
$\lambda_i\times\omega=0$ on $\Gamma_i$ for each $i=1,\ldots,N.$
\end{theorem}

\begin{example}\label{remark:to main theorem polygonial in re2}
As a special case we obtain  Gaffney inequality with the classical boundary conditions in polygonial domains.
\end{example}

The proof of Theorem \ref{theorem:lambda wedge omega polygonial in re2} is essentially the same as the corresponding result for globally Lipschitz $\lambda:$  only the approximation result, i.e. the analogy to Proposition \ref{prop:approx of lambda times omega zero by smooth} has to be adapted. This is done in the next proposition.

\begin{proposition}
\label{prop:approx of lambda times omega zero by smooth: polgynial dom}
Let $\Omega$ and $\lambda$ be as in Assumption \ref{Assumption lambda not glob lip}.
Suppose $\omega\in H^1(\Omega)^2$ is such that $\lambda\times\omega=0$ on $\delomega$. Then there exists a sequence $\{\omega^k\}_{k\in\mathbb{N}}\subset C^{0,1}(\Omegabar)^2$ such that for $k\to\infty$
$$
  \omega^k\to\omega\in H^1(\Omega)^2
  \quad\text{ and }\quad
  \lambda\times\omega^k=0\quad\text{ on }\delomega\quad\text{ for all }k.
$$
\end{proposition}

\begin{proof}
\textit{Step 1.} We first prove the following claim:
For every $x_0\in\delomega$ there exists a neighborhood $W\subset\re^2$ of $x_0$ such that for all $\omega\in H^1(\Omega)$ satisfying
\begin{equation}
 \label{equation:lemma approx local in poly omegaassumption}
  \operatorname{supp}(\omega)\subset W\quad\text{ and }\quad \lambda\times \omega=0\quad\text{ on }\delomega,
\end{equation}
there exists a sequence $\{\omega^k\}_{k\in\mathbb{N}}\subset C^{0,1}\left(\overline{\Omega\cap W}\right)^2$ such that
$$
  \omega^k\to \omega\quad\text{ in }H^1\left(\Omega\cap W\right)^2\quad\text{ and }\quad \lambda\times \omega^k=0\quad\text{ on }\delomega\cap W\text{ for all }k.
$$
The proof of this claim is the same as the proof of Proposition \ref{prop:approx of lambda times omega zero by smooth} if $x_0$ is not a vertex and so we can assume that $x_0=\overline{\Gamma}_i\cap\overline{\Gamma}_{i+1}$ is a vertex. Then we distinguish two cases.
\smallskip

\textit{Case 1.} We assume that $\lambda^i(x_0)$ and $\lambda^{i+1}(x_0)$ are linearly dependent.
Since the boundary condition $\lambda\times\omega=0$ is invariant under scaling or sign change of $\lambda,$ and neither $\lambda^i$ nor $\lambda^{i+1}$ vanish, we can assume that $\lambda^i(x_0)=\lambda^{i+1}(x_0).$ But then $\lambda$ is Lipschitz in $\overline{\Gamma}_i\cup\overline{\Gamma}_{i+1}$ and thus we can again proceed as in Proposition \ref{prop:approx of lambda times omega zero by smooth}.
\smallskip

\textit{Case 2.} We assume that $\det(\lambda^i(x_0)|\lambda^{i+1}(x_0))\neq 0.$
In this case we extend both $\lambda^i$ and $\lambda^{i+1}$ separately to $C^{0,1}$ vector fields defined in $\re^2.$ By continuity, there exists a neighborhood $W$ of $x_0$ such that $\det(\lambda^i|\lambda^{i+1})\neq 0$ in $\overline{W}.$
Let $\omega$ be a vector field satisying \eqref{equation:lemma approx local in poly omegaassumption}. Define the two functions
$$
 p=\lambda^i\times \omega\in H^1(\Omega),\quad\text{ and }\quad q=\lambda^{i+1}\times\omega\in H^1(\Omega),
$$
which can also be written in the matrix form
$$
  \left(\begin{array}{c}
         p \\ q
        \end{array}\right)
        =
        \left(\begin{array}{cc}
        -\lambda_2^{i} & \lambda_1^i
        \smallskip \\
        -\lambda_2^{i+1} & \lambda_1^{i+1}\end{array}\right)
        \left(\begin{array}{c}
         \omega_1 \\ \omega_2
        \end{array}\right)=M\,\left(\begin{array}{c}
         \omega_1 \\ \omega_2
        \end{array}\right).
$$
By an extension theorem (see Bernard \cite{Bernard J.M.} or Theorem 1.6.1 of Grisvard \cite{Grisvard Singularities in B} for polygonial domains) there exists sequences $\{p^k\}_{k\in \mathbb{N}},$ $\{q^k\}_{k\in \mathbb{N}}\in C^1(\Omegabar)$ such that both $p^k$ (respectively $q^k$) converges to $p$ (respectively $q$) in $H^1(\Omega)$ and
$$
  p^k=0\quad\text{ on }\Gamma_i\quad\text{ and }\quad q^k=0\quad\text{ on }\Gamma_{i+1}\qquad \text{ for all }k\in\mathbb{N}.
$$
Since $\det(\lambda^i|\lambda^{i+1})\neq 0$ on $\overline{W},$ we can define $\omega^k\in C^{0,1}(\overline{W\cap \Omega})$ by
$$
  \omega^k=M^{-1}\left(\begin{array}{c}
                        p^k \\ q^k
                       \end{array}
\right).
$$
Note that $\lambda^i\times\omega^k=p^k,$ respectively $\lambda^{i+1}\times\omega^k=q^k.$
It straightforward to check that $\omega^k$ has all the desired properties claimed by Step 1.
\smallskip
\\

\textit{Step 2.} We finally conclude exactly as in Step 2 of the proof of Proposition \ref{prop:approx of lambda times omega zero by smooth}.
\end{proof}

\smallskip

We now prove the main theorem of this section.

\smallskip

\begin{proof}[Proof of Theorem \ref{theorem:lambda wedge omega polygonial in re2}.]
Since $\Omega$ is a Lipschitz domain we can use partial integration and obtain  that
\begin{align*}
  \intomega\left(|\curl \omega|^2+|\diw \omega|^2-|\nabla\omega|^2\right)
  =&
  \int_{\delomega}\omega_1\left(\nu_1\frac{\partial \omega_2}{\partial x_2}-\nu_2\frac{\partial\omega_2}{\partial x_1}\right)
  -\int_{\delomega}\omega_2\left( \nu_1\frac{\partial \omega_1}{\partial x_2}-\nu_2\frac{\partial\omega_1}{\partial x_1}\right)
  \smallskip \\
  =&
  \sum_{i=1}^N\left[\int_{\Gamma_i}\omega_1\left(\nu_1\frac{\partial \omega_2}{\partial x_2}-\nu_2\frac{\partial\omega_2}{\partial x_1}\right)
  -\int_{\Gamma_i}\omega_2\left( \nu_1\frac{\partial \omega_1}{\partial x_2}-\nu_2\frac{\partial\omega_1}{\partial x_1}\right)\right]
\end{align*}
holds for any $\omega\in C^{0,1}\left(\Omegabar\right)^2,$ where the first equality is exactly as in Step 1 of the first proof of Proposition \ref{proposition:lambda wedge omega for vector fields}). We now proceed
as in Step 2 of the first proof of Proposition \ref{proposition:lambda wedge omega for vector fields}, working on each $\Gamma_i$ separately:
Using that each $\Gamma_i$ is a $C^{0,1}$ curve and that $\lambda^i$ does not vanish on $\Gamma_i$, one obtains that there exists a constant $C_1=C_1(\Omega,\lambda)>0$ such that
$$
  \intomega\left(|\curl\omega|^2+|\diw\omega|^2-|\nabla\omega|^2\right)\geq -C_1\intdelomega|\omega|^2
$$
for all $\omega\in C^{0,1}\left(\Omegabar\right)^2$ satisfying $\lambda\times\omega=0$ on $\delomega.$ This proves the Theorem for $C^{0,1}$ vector fields $\omega.$ The general case follows from Proposition \ref{prop:approx of lambda times omega zero by smooth: polgynial dom}
\end{proof}

\section{Counterexamples}

In view of Theorems \ref{theorem:lambda wedge omega for vector fields}, \ref{theorem:lambda wedge omega polygonial in re2} and the classical boundary condition $\langle\nu;\omega\rangle=0,$ 
one could expect that if $n\geq 3$ we also have a Gaffney inequality under the boundary condition
$$
  \langle\lambda;\omega\rangle=0\quad\text{ on }\delomega,
$$
if $\lambda$ does not vanish on $\delomega.$ This is however not true as shown by the following simple example.

\begin{example}\label{example:counter to lambda scal omega if n is 3}
Let $\Omega\subset\re^3$ be any bounded open smooth set and $\lambda=(0,0,1).$ Then there exists no constant $C=C(\Omega,\lambda)$ such that
$$
  \intomega|\nabla\omega|^2\leq C\intomega \left(|\curl\omega|^2+|\diw\omega|^2+|\omega|^2\right)
$$
for all $\omega\in C^2(\Omegabar;\re^n)$ satisfying $\langle\lambda;\omega\rangle=0$ on $\delomega.$ To see this take
$$
  \omega(x)=\left(e^{nx_1}\cos(nx_2),-e^{nx_1}\sin(nx_2),0\right).
$$
Then one easily verifies that $\diw\omega=0,$ $\curl\omega=0,$ $|\nabla\omega(x)|^2=2n^2e^{2nx_1}$ and $|\omega(x)|^2=e^{2nx_1}.$
Hence, as in \eqref{eq:intro:counterexample}, Gaffney inequality cannot hold.
\end{example}

The question also arises whether Theorem \ref{theorem:lambda wedge omega for vector fields} generalizes to differential forms of higher order (identifying vector fields with $1$-forms). This is also not true. More precisely we have the following counterexample for $2$-forms.

\begin{example}Let $n\geq 3,$ $\Omega\subset\re^n$ be a bounded open smooth set. Then there exists no constant $C=C(\Omega)$ such that
$$
  \intomega|\nabla\omega|^2\leq C\intomega\left(|d\omega|^2+|\delta\omega|^2+|\omega|^2\right)
$$
for all $\omega\in C^2(\Omegabar;\Lambda^2)$ such that $dx^3\wedge\omega=0$ on $\delomega.$ To see this take
$$
  \omega=e^{nx_1}\cos(nx_2)dx^1\wedge dx^3+e^{nx_1}\sin(nx_2)dx^2\wedge dx^3.
$$
One can verify that $d\omega=0$ and $\delta\omega=0.$ Thus one concludes exactly as in Example \ref{example:counter to lambda scal omega if n is 3}.
\end{example}
\smallskip

\noindent\textbf{Acknowledgements} The first author was supported by Chilean FONDECYT Iniciaci\'on grant nr. 11150017. 
He would also like to thank Olivier Kneuss and Wladimir Neves for the kind invitation and hospitality at the Universidad Federal de Rio de Janeiro in June 2016, during which a relevant part of this work was finalized. The third author, Dhanya R., was supported by INSPIRE faculty fellowship (DST/INSPIRE/04/2015/003221) when a part of this work was carried out.

Moreover we would like to thank Martin Werner Licht, whose question partially motivated this research and who pointed out to us the connenction with the  references \cite{Arnold Falk Winther} and \cite{Bonizzoni Buffa Nobile}.

\end{document}